\documentclass[11pt]{article}
\clearpage{}\usepackage{fullpage}
\usepackage{amsmath,amsthm,amsfonts,dsfont}
\usepackage{amssymb,latexsym,graphicx}
\usepackage{mathtools}
\usepackage{hyperref}
\usepackage{xcolor}
\hypersetup{
	colorlinks,
	linkcolor={red!75!black},
	citecolor={blue!75!black},
	urlcolor={blue!75!black}
}

\newtheorem{theorem}{Theorem}[section]

\newtheorem{corollary}[theorem]{Corollary}

\newcommand{\R}{\ensuremath{\mathbb{R}}}
\newcommand{\Z}{\ensuremath{\mathbb{Z}}}

\newcommand{\lat}{\mathcal{L}}

\renewcommand{\epsilon}{\varepsilon}

\renewcommand{\vec}[1]{\ensuremath{\boldsymbol{#1}}}
\newcommand{\basis}{\ensuremath{\mathbf{B}}}

\DeclarePairedDelimiter\inner{\langle}{\rangle}

\mathtoolsset{centercolon}\clearpage{}

\usepackage[hyperpageref]{backref}

\begin{document}
	
		\title{An improved constant in Banaszczyk's transference theorem}
		\author{
			Divesh Aggarwal\thanks{This research was partially funded by the Singapore Ministry of Education and the National Research Foundation under grant R-710-000-012-135}\\
			National University of Singapore\\
			\texttt{dcsdiva@nus.edu.sg}
			\and
			Noah Stephens-Davidowitz\\
			Massachusetts Institute of Technology\\
			\texttt{noahsd@gmail.com}
		}
		\date{}
		\maketitle
		
		\begin{abstract}
			We show that
			\[
				\mu(\lat) \lambda_1(\lat^*) < \big( 0.1275 + o(1) \big) \cdot n
				\; ,
			\]
			where $\mu(\lat)$ is the covering radius of an $n$-dimensional lattice $\lat \subset \R^n$ and $\lambda_1(\lat^*)$ is the length of the shortest non-zero vector in the dual lattice $\lat^*$. This improves on Banaszczyk's celebrated transference theorem (Math. Annal., 1993) by about 20\%.
			
			Our proof follows Banaszczyk exactly, except in one step, where we replace a Fourier-analytic bound on the discrete Gaussian mass with a slightly stronger bound based on packing. The packing-based bound that we use was already proven by Aggarwal, Dadush, Regev, and Stephens-Davidowitz (STOC, 2015) in a very different context. Our contribution is therefore simply the observation that this implies a better transference theorem.
		\end{abstract}
		
		\section{Introduction}

		 A lattice $\lat \subset \R^n$ is the set of integer linear combinations of linearly independent basis vectors $\basis = (\vec{b}_1,\ldots, \vec{b}_n)$. I.e.,
		 \[
		 \lat := \{z_1 \vec{b}_1 + \cdots + z_n \vec{b}_n \ : \ z_i \in \Z \}
		 \; .
		 \]
		 The \emph{dual} lattice $\lat^*$ is the set of vectors that have integer inner product with all elements in $\lat$. I.e.,
		 \[
		 \lat^* := \{\vec{w} \in \R^n \ : \ \forall \vec{y} \in \lat,\ \inner{\vec{w}, \vec{y}} \in \Z \}
		 \; .
		 \]
		
		A \emph{transference theorem} relates the geometry of the primal lattice $\lat$ to that of the dual lattice $\lat^*$. For example, the first minimum
		\[
		\lambda_1(\lat) := \min_{\vec{y} \in \lat_{\neq \vec0}} \|\vec{y}\|
		\; 
		\]
		is the minimal (Euclidean) norm of a non-zero lattice vector, and the \emph{covering radius} 
		\[
		\mu(\lat) := \max_{\vec{t} \in \R^n} \min_{\vec{y}\in \lat} \|\vec{y} - \vec{t}\|
		\; 
		\]
		is the maximal distance from any point in space to the lattice. Banaszczyk's celebrated transference theorem states that the covering radius of $\lat$ is rather closely related to the first minimum of the dual lattice, as follows.
		
		\begin{theorem}[\cite{BanNewBounds93}]
		\label{thm:transference_banaszczyk}
		For any lattice $\lat \subset \R^n$,
		\[
		\frac{1}{2} \leq \mu(\lat) \lambda_1(\lat^*) \leq \Big(\frac{1}{2\pi} + o(1) \Big) \cdot n
		\; .
		\]
		\end{theorem}
	
		(Here and elsewhere, we write $o(1)$ for an unspecified function that approaches zero as $n$ grows. Banaszczyk actually formally proved a slightly weaker bound, but he noted at the end of his paper that his proof yields Theorem~\ref{thm:transference_banaszczyk}. See, e.g.,~\cite{MSKissingNumbers19}.)
	
		We are interested in the upper bound in Theorem~\ref{thm:transference_banaszczyk}, and we include the simple lower bound only for completeness. I.e., we are interested in the quantity
		\[
		T_n := \frac{1}{n} \cdot \sup_{\lat \subset \R^n} \mu(\lat) \lambda_1(\lat^*)
		\; ,
		\]
		where the supremum is taken over all lattices in $n$ dimensions. Theorem~\ref{thm:transference_banaszczyk} shows that $T_n < 1/(2\pi) + o(1) \approx 0.159$, 
		and it is known that
		\begin{equation}
			\label{eq:random_lattice}
		 	T_n > \frac{1}{2\pi e} - o(1) \approx 0.059 
			\; ,
		\end{equation}
		so that $T_n$ is known up to a constant factor.
		(Eq.~\eqref{eq:random_lattice} follows, e.g., from~\cite{SieMeanValue45}.)
		
		Our main result is the following refinement of Theorem~\ref{thm:better_transference}.
		
		\begin{theorem}
			\label{thm:better_transference}
			For any lattice $\lat \subset \R^n$, we have
			\[
			\frac{1}{2} \leq \mu(\lat) \lambda_1(\lat^*) < \big( 0.1275 + o(1) \big) \cdot n
			\]
			I.e.,
			\[
			T_n < 0.1275 + o(1)
			\; .
			\]
		\end{theorem}
		
		Theorem~\ref{thm:better_transference} is a roughly 20\% improvement over Banaszczyk's Theorem~\ref{thm:transference_banaszczyk}, but still rather far from matching the lower bound in Eq.~\eqref{eq:random_lattice}. In fact, we prove a potentially stronger bound of
		\[
		T_n < \frac{2^{\beta_n }}{2\pi \sqrt{e}} + o(1)
		\; ,
		\]
		where $\beta_n$ is a certain geometric quantity known to satisfy
		\[
		0.0219 - o(1) < \beta_n < 0.401 + o(1)
		\; .
		\]
		See Eq.~\eqref{eq:beta}.
		
		\section{Banaszczyk's original proof}
		
		Like Banaszczyk's original proof, our proof of Theorem~\ref{thm:better_transference} works by studying the Gaussian mass
		\[
		\rho_{s, r}(\lat - \vec{t}) := \sum_{\stackrel{\vec{y} \in \lat}{\|\vec{y} - \vec{t}\| \geq r}} \exp(-\pi \|\vec{y} - \vec{t}\|^2/s^2)
		\]
		for a lattice $\lat \subset \R^n$, parameter $s > 0$, shift vector $\vec{t}\in \R^n$, and radius $r \geq 0$. When $r = 0$, we simply write $\rho_s(\lat - \vec{t})$. In particular, notice that the covering radius $\mu(\lat)$ is the maximal radius $r > 0$  such that $\rho_{s,r}(\lat - \vec{t}) = \rho_{s}(\lat - \vec{t})$ for some $\vec{t}\in \R^n$. To obtain a bound $\mu(\lat) < r$, it therefore suffices to prove that	
		\[
		\rho_{s,r}(\lat - \vec{t}) < \rho_{s}(\lat - \vec{t})
		\]
		for some parameter $s > 0$ and all $\vec{t}\in \R^n$.
		
		To that end, using the language and notation of~\cite{MRWorstcaseAveragecase07}, we define the \emph{smoothing parameter} $\eta = \eta(\lat) > 0$ to be the unique parameter satisfying $\rho_{1/\eta}(\lat^*) = 3/2$.\footnote{There is nothing particularly special about the constant $3/2$ in this definition. Any constant strictly between $1$ and $2$ would suffice for our purposes, though our choice of constant gives a slightly cleaner proof.}
		Using the Poisson Summation Formula, Banaszczyk showed that
		\begin{equation}
		\label{eq:smooth}
		\rho_{s}(\lat)/3 < \rho_{s}(\lat - \vec{t}) \leq \rho_s(\lat)
		\; 
		\end{equation}
		for any $s \geq \eta(\lat)$ and $\vec{t} \in \R^n$.
		
		So, for such a parameter $s \geq \eta(\lat)$ and a suitable radius $r > 0$, we wish to show that $\rho_{s,r}(\lat - \vec{t}) \leq \rho_s(\lat)/3$ for all $\vec{t}\in \R^n$. Intuitively, we expect this to be true when $r$ is large relative to $s$. Indeed, Banaszczyk's celebrated tail bound says exactly this. Using the Poisson Summation Formula again, he showed that
		\begin{equation}
			\label{eq:banaszczyk_tail}
			\rho_{s,r}(\lat - \vec{t}) \leq \rho_s(\lat)/3
		\end{equation}
		for $r \geq  C_n\sqrt{n} \cdot s $ where  $C_n = 1/\sqrt{2\pi}+o(1)$. (Banaszczyk actually proved a more general bound that holds for all $r \geq \sqrt{n/(2\pi)} \cdot s$, but we will only need this special case.) Therefore,
		\begin{equation}
		\label{eq:mu_eta_intro}
		\mu(\lat) < C_n \sqrt{n} \cdot \eta(\lat)
		\; .
		\end{equation}
		
		We note that the continuous Gaussian with parameter $s$ has mass concentrated in a thin shell of radius roughly $C_n \sqrt{n} s$. For sufficiently large $s$, the discrete Gaussian mass $\rho_s(\lat - \vec{t})$ is similarly concentrated. In particular, Eq.~\eqref{eq:banaszczyk_tail} is tight up to a constant when $s \geq \eta(\lat)$. Therefore, it seems difficult (though perhaps not impossible) to improve upon this step in Banaszczyk's proof.\footnote{The authors do not know of an example where Eq.~\eqref{eq:mu_eta_intro} is tight. So, it is conceivable that one could improve Eq.~\eqref{eq:mu_eta_intro} substantially without improving on Eq.~\eqref{eq:banaszczyk_tail} much. This seems to require a very fine understanding of the behavior of the discrete Gaussian at small radii.}
		
		The last step in the proof (as presented here) is where we will diverge from Banaszczyk, but it will still be instructive to complete Banaszczyk's original proof. 
		To do so, Banaszczyk applied his tail bound once more to bound $\eta(\lat)$ in terms of $1/\lambda_1(\lat^*)$. In particular, notice that $\rho_{1/s}(\lat^*) = 1 + \rho_{1/s, \lambda_1(\lat^*)}(\lat^*)$. Therefore, if $s \geq C_n \sqrt{n}/\lambda_1(\lat^*)$, Eq.~\eqref{eq:banaszczyk_tail} implies that $\rho_{1/s}(\lat^*) \leq 1 + \rho_{1/s}(\lat^*)/3$. Rearranging gives $\rho_{1/s}(\lat^*) \leq 3/2$, i.e.,
		\begin{equation}
			\label{eq:eta_lambda_intro}
			\eta(\lat) \leq C_n \sqrt{n}/\lambda_1(\lat^*)
			\; .
		\end{equation}
		Combining Eqs.~\eqref{eq:mu_eta_intro} and~\eqref{eq:eta_lambda_intro} yields Theorem~\ref{thm:transference_banaszczyk}, $\mu(\lat) \lambda_1(\lat^*) \leq C_n^2 \cdot n$.
		
		While Banaszczyk's tail bound Eq.~\eqref{eq:banaszczyk_tail} is quite tight when the parameter $s $ is sufficiently large, $s \geq \eta(\lat)$, it is not necessarily tight for smaller parameters. Indeed, in the last step above, we specifically chose such a small parameter that nearly all of the Gaussian mass is concentrated on $\vec0$. For such small parameters, Eq.~\eqref{eq:banaszczyk_tail} is in fact loose, as we will show in the next section. By improving on the tail bound in this special case, we will improve Eq.~\eqref{eq:eta_lambda_intro}, thus obtaining the better transference theorem in Theorem~\ref{thm:better_transference}.
		
		\section{Proof of Theorem~\ref{thm:better_transference}}
		
		For a lattice $\lat \subset \R^n$ and $\alpha \geq 1$, let 
		\[
			N_\alpha(\lat) := |\{ \vec{y}\in \lat \ : \ 0 < \|\vec{y}\| \leq \alpha \lambda_1(\lat) \}|
		\]
		be the number of non-zero lattice points inside a ball of radius $\alpha \lambda_1(\lat)$. 
		E.g., $N_1(\lat)$ is the kissing number of $\lat$, the number of shortest non-zero vectors.
		
		Intuitively, for large $\alpha$, we expect $N_\alpha(\lat)$ to be proportional to the volume of the ball of radius $\alpha \lambda_1(\lat)$, and therefore to be proportional to $\alpha^n$. Indeed, for a random lattice $\lat \subset \R^n$ under the Haar measure, $N_\alpha(\lat)$ is concentrated closely around $\alpha^n$. (See~\cite{SieMeanValue45}.)
		It is therefore natural to define
		\begin{equation}
		\label{eq:beta}
		\beta_n := \frac{1}{n}\cdot \log \sup_{\stackrel{\lat \subset \R^n}{\alpha \geq 1}}  \frac{N_\alpha(\lat)}{\alpha^n} 
		\; ,
		\end{equation}
		where by convention we take the logarithm base two (here and below).
		Notice that $\beta_n$ measures how much this volume heuristic can underestimate $N_\alpha$. (Until recently, it was not even clear whether $\beta_n$ is bounded away from zero. But, Vlăduţ recently proved the existence of lattices with exponentially large kissing number, which implies that $\beta_n$ is in fact bounded below by some constant. Specifically, $\beta_n > 0.0219 - o(1)$~\cite{VlaLatticesExponentially19}.) 
		
		Upper bounds on $\beta_n$ and $N_\alpha$ are quite well studied. For example, Eq.~\eqref{eq:banaszczyk_tail} implies that $\log N_1(\lat) < (\log(e)/2+o(1))\cdot  n$, and the more general tail bound in~\cite{BanNewBounds93} implies that $\beta_n < \log (e)/2 + o(1)$. Indeed, Banaszczyk's original transference theorem essentially follows from this bound.
		
		However, the best asymptotic upper bound known is due to Kabatjanski{\u\i} and Leven{\v{s}}te{\u\i}n~\cite{KLBoundsPackings78}.\footnote{Kabatjanski{\u\i} and Leven{\v{s}}te{\u\i}n formally only showed a bound on $N_1(\lat)$, but this can easily be extended to a bound on $\beta_n$. See~\cite[Lemma 3]{PSSolvingShortest09}.} In particular, they show that
		\begin{equation}
			\label{eq:KL}
			\beta_n < 0.401 + o(1)
			\; .
		\end{equation}
		We simply observe that such a bound on $\beta_n$ yields improvements to Eq.~\eqref{eq:eta_lambda_intro}. In fact, the following theorem already appeared in~\cite{ADRSSolvingShortest15} in a very different context. At the time, we did not recognize the relevance to transference.
		
		\begin{theorem}[{\cite[Lemma 4.2]{ADRSSolvingShortest15}}]
			\label{thm:KL_mass}
			For any lattice $\lat \subset \R^n$ and any parameter $s > 0$,
			\[
			\rho_s(\lat) < 1 + \Big( \frac{2^{2\beta_n + o(1)}s^2 n}{2\pi  e \lambda_1(\lat)^2}\Big)^{n/2}  
			\; .
			\]
		\end{theorem}
		\begin{proof}
			We have
			\begin{align*}
			\rho_s(\lat) 
			&= 1 +  \frac{2\pi \lambda_1(\lat)^2}{s^2} \cdot \int_{1}^\infty N_\alpha(\lat) \cdot \alpha \exp(-\pi \alpha^2 \lambda_1(\lat)^2/s^2) {\rm d} \alpha\\
			&\leq 1 + \frac{2\pi \lambda_1(\lat)^2 }{s^2} \cdot 2^{\beta_n n} \cdot \int_{1}^\infty \alpha^{n+1} \cdot \exp(-\pi \alpha^2 \lambda_1(\lat)^2/s^2) {\rm d} \alpha\\
			&< 1 + \frac{2\pi \lambda_1(\lat)^2 }{s^2} \cdot 2^{\beta_n n} \cdot \int_{0}^\infty \alpha^{n+1} \cdot \exp(-\pi \alpha^2 \lambda_1(\lat)^2/s^2) {\rm d} \alpha\\
			&= 1 + \Big( \frac{2^{2\beta_n}s^2}{\pi \lambda_1(\lat)^2}\Big)^{n/2} \cdot \Gamma(n/2 + 1)\\
			&< 1 + \Big( \frac{2^{2\beta_n + o(1)}s^2 n}{2\pi  e \lambda_1(\lat)^2}\Big)^{n/2}  
			\; ,
			\end{align*}
			as needed.
		\end{proof}
	
	\begin{corollary}
		For any lattice $\lat \subset \R^n$,
			\begin{equation}
		\label{eq:better_lambda_1}
		\eta(\lat) < \Big(\frac{2^{\beta_n}}{\sqrt{2\pi e}} + o(1) \Big)\cdot  \frac{\sqrt{n}}{\lambda_1(\lat^*)} < \Big(\frac{2^{0.401}}{\sqrt{2\pi e}} + o(1) \Big)\cdot  \frac{\sqrt{n}}{\lambda_1(\lat^*)}
		\; .
		\end{equation}
	\end{corollary}
\begin{proof}
	Taking $s > \sqrt{n/(2\pi e)} \cdot 2^{\beta_n + o(1)} /\lambda_1(\lat^*)$ in Theorem~\ref{thm:KL_mass} yields $\rho_{1/s}(\lat^*) < 3/2$. I.e., $\eta(\lat) < s$, as needed.
\end{proof}

	Theorem~\ref{thm:better_transference} then follows by combining Eqs.~\eqref{eq:mu_eta_intro} and~\eqref{eq:better_lambda_1}.

		\bibliographystyle{alpha}

\end{document}